\documentclass[11pt,reqno]{amsart} 
\usepackage[utf8]{inputenc}
\usepackage{amsmath, amsthm, amsfonts,color,url,enumitem,subcaption,amssymb}
\usepackage{cancel}
\usepackage{wrapfig,hyperref}
\usepackage{physics}
\usepackage{mathtools}
\usepackage{fullpage}
\usepackage{color}
\usepackage{thmtools}
\usepackage{thm-restate}
\usepackage{hyperref}
\usepackage{cleveref}
\usepackage[normalem]{ulem}
\usepackage{tikz}
\usepackage{dsfont}
\usepackage{physics}
\usepackage{floatrow}
\usepackage{url}
\usepackage{graphicx}
 \usepackage{diagbox}
\newfloatcommand{capbtabbox}{table}[][\FBwidth]

\numberwithin{equation}{section}
\newcounter{dummy} \numberwithin{dummy}{section}

\theoremstyle{definition} 
\newtheorem{theorem}[dummy]{Theorem}
\newtheorem{definition}[dummy]{Definition}
\newtheorem{lemma}[dummy]{Lemma}
\newtheorem{example}[dummy]{Example}

\newtheorem{remark}[dummy]{Remark}

\newtheorem{proposition}[dummy]{Proposition}

\newcommand{\unif}[1]{{\rm Uni}_{#1}}
\newcommand{\PP}{\mathbb{P}}
\newcommand{\TV}{\mathrm{TV}}
\newcommand{\PF}[1]{\mathrm{PF}_{#1}}

\title{Probabilistic $(m,n)$-Parking Functions}
\author[Harris]{Pamela E. Harris}\thanks{P. E.~Harris was supported in part by an award from the Simons Travel Support for Mathematicians program.}
\address[P. E.~Harris]{Department of Mathematical Sciences, University of Wisconsin, Milwaukee, WI, U.S.A.}
\email{\textcolor{blue}{\href{mailto:peharris@uwm.edu}{peharris@uwm.edu}}}

\author[Ribeiro]{Rodrigo Ribeiro}
\author[M. Yin]{Mei Yin}\thanks{M.~Yin was supported in part by the University of Denver's Professional Research Opportunities for Faculty Fund 80369-145601 and an award from the Simons Travel Support for Mathematicians program.}
\address[R. Ribeiro and M. Yin]{Department of Mathematics, University of Denver, Denver, CO, U.S.A.}
\email{\textcolor{blue}{\href{mailto:rodrigo.ribeiro@du.edu}{rodrigo.ribeiro@du.edu}}, \textcolor{blue}{\href{mailto:mei.yin@du.edu}{mei.yin@du.edu}}}

\begin{document}
\begin{abstract}
In this article, we establish new results on the probabilistic parking model (introduced by Durm\'ic, Han, Harris, Ribeiro, and Yin) with $m$ cars and $n$ parking spots and probability parameter $p\in[0,1]$. For any $ m \leq n$ and $p \in [0,1]$, we study the parking preference of the last car, denoted $a_m$, and determine the conditional distribution of $a_m$ and compute its expected value. We show that both formulas depict explicit dependence on the probability parameter $p$.
We study the case where $m = cn $ for some $ 0 < c < 1 $ and investigate the asymptotic behavior and show that the presence of ``extra spots'' on the street significantly affects the rate at which the conditional distribution of $ a_m $ converges to the uniform distribution on $[n]$. 
Even for small $ \varepsilon = 1 - c $, an $ \varepsilon $-proportion of extra spots reduces the convergence rate from $ 1/\sqrt{n} $ to $ 1/n $ when $ p \neq 1/2 $.
Additionally, we examine how the convergence rate depends on $c$, while keeping $n$ and $p$ fixed. We establish that as $c$ approaches zero, the total variation distance between the conditional distribution of $a_m$ and the uniform distribution on $[n]$ decreases at least linearly in $c$. 
\end{abstract}
\maketitle

\section{Introduction}\label{intro}
A parking function of length $n\in\mathbb{N}=\{1,2,3,\ldots\}$ is a tuple $(a_1,a_2,\ldots,a_n)\in[n]^n$, where $[n]=\{1,2,\ldots,n\}$, such that its nondecreasing rearrangement $(b_1,b_2,\ldots,b_n)$ satisfies $b_i\leq i$ for all $i\in n$. As their name indicates, parking functions, introduced in the literature by Konheim and Weiss \cite{konheimWeiss}, can also be described via the following parking process: 
There are $n$ cars in queue to enter a one-way street with $n$ parking spots numbered in sequence from $1$ to $n$. 
Car $i$ has a parking preference $a_i\in[n]$, and we collect these preferences in a \emph{preference list} $\alpha=(a_1,a_2,\ldots,a_n)\in[n]^n$. 
Each car enters the street and proceeds to their preferred parking spot, parking there if the spot is unoccupied. 
If a car finds their preferred spot occupied, then they continue down the street parking in the next available spot, if such a spot exists. 
Given the preference list $\alpha$, if all cars are able to park within the $n$ parking spots on the street under this classical parking protocol, then we say that the preference list $\alpha$ is a 
\emph{parking function} of length $n$. 
We let $\PF{n}$ denote the set of all parking functions of length $n$ 
and recall that $|\PF{n}|=(n+1)^{n-1}$, for a proof see \cite{konheimWeiss}.

There are many generalizations of parking functions, and they often arise via changes to the cars, to their preferences, or to the parking protocol involved. 
Such generalizations include having more spots than cars, cars having varying lengths (parking sequences and assortments), cars having an interval or subset of preferences (interval and subset parking functions), cars being able to first seek for parking backwards when finding their preference occupied (vacillating and $k$-Naples parking functions), or cars being bumped out of their preference by a later car in the queue (MVP parking functions) \cite{aguilarfraga2023interval,armstrong2016rational, 
BrandtUnitInterval,
kNaplesthroughCatalan,
ChenAssortments,
bib:NaplesPF, colaric2020interval,
countingKnaples, fang2024vacillatingparkingfunctions,franks2023counting,MVP,SpiroSubset}. 
Parking functions are also well-connected to numerous areas of mathematics and make appearances in the study of: diagonal harmonics, ideal states in the Tower of Hanoi game, the \texttt{QuickSort} algorithm, Boolean intervals in the weak order on the symmetric group, facets of the permutohedron, and computing volumes of flow polytopes 
\cite{Aguillon2022OnPF,
Caracol,
MR3724106,
elder2025parking,
fang2024vacillatingparkingfunctions,
garcia2024defectiveparkingfunctionsyoung,
MR3551576,
quicksort,
Harris2023LuckyCA,
HicksDiagonalHarmonics,
unit_perm,
PitmanStanley,
MR1902680}. 
We point the interested reader to the work of Yan \cite{yan2015parking} providing a survey of parking function results, and that of Carlson, Christensen, Harris, Jones, and Ramos Rodr\'iguez \cite{carlson2020parking} with many open problems related to parking functions.

Although parking functions and their generalizations are well-studied, there has been less work done on considering probabilistic parking protocols, see \cite{DHHRY,Amanda,Trevino}. 
In \cite{DHHRY}, we investigate the following probabilistic scenario for parking $n$ cars on a one-way street with $n$ spots: Fix $p \in [0, 1]$ and consider a coin which flips to heads with probability $p$ and tails with probability $1-p$. 
Our probabilistic parking protocol proceeds as follows: If a car arrives at its preferred spot and finds it unoccupied it parks there. If instead the spot is occupied, then the driver tosses the biased coin. If the coin lands on heads, with probability $p$, the driver continues moving forward in the street. However, if the coin lands on tails, with probability $1-p$, the car moves backwards and tries to find an unoccupied parking spot.
With this probabilistic parking function protocol, we determined the likelihood that a preference
list $\alpha \in [n]^n$ is a parking function. 
We also explored the properties of parking functions $\alpha$, demonstrating the effect of the parameter $p$ on the parking protocol.

Motivated by that initial work \cite{DHHRY}, the following natural questions arise: \emph{What if we are parking $m \leq n$ cars on a one-way street with $n$ spots? How do parking statistics depend on the probabilistic parameter $p$ in this more general case? What new implications arise based on the change to the number of cars?}  
We answer these questions here and we summarize our finds below, but first we set some notation and make some initial remarks.

For $m\leq n$, it is customary to let $\PF{m,n}$ denote the set of $(m,n)$-parking functions which are the preference lists $(a_1,a_2,\ldots,a_m)\in[n]^m$ for $m$ cars parking on a one-way street with $n$ parking spots, that under the classical parking protocol allow all cars to park. 
Konheim and Weiss \cite{konheimWeiss} establish that $|\PF{m,n}|=(n-m+1)(n+1)^{m-1}$.
Due to the probabilistic nature of our parking model, in our setting, all $\alpha \in [n]^m$ have a positive probability of allowing all of the cars to park, and hence, of being an $(m,n)$-parking function. 
Thus, throughout we write $\alpha \in \PF{m, n}$ to depict the situation that $m$ cars with preference list $\alpha$ park on $n$ spots. 
In \cite[Theorem 2]{DHHRY}, using circular symmetry ideas, we establish that the likelihood that a preference list $\alpha \in [n]^m$ is a parking function does not depend on $p$, which is why the parameter does not arise in the notation $\PF{m,n}$.
Parking $m \leq n$ cars instead of $n$ cars on $n$ spots opens up many new possibilities, and as such, the results in this paper and their proofs are more complicated than the special case of $m=n$ which we explored in \cite{DHHRY}. 
Hence,  our contributions delve deeper and advance our understanding of the implications of the probabilistic model.

We now give an overview of our results. 
\begin{enumerate}
\item As in \cite{DHHRY}, we focus on the parking preference of the last car $a_m$ and explore its statistical properties. 
In Theorem \ref{component} we give the conditional distribution of $a_m$, in Theorem \ref{mean} we calculate its expected value, and in Theorem \ref{thm:tvd} we study the rate of convergence of the conditional distribution of $a_m$ to the uniform distribution on $[n]$. 
These results are valid for any $m\leq n$ and $p \in [0, 1]$. 

\item An important feature of the probabilistic parking model is the \emph{parking symmetry}: Having $m$ cars enter the street from left to right with preference list $\alpha=(a_1, \dots, a_m)$ parking under protocol with parameter $p$ depicts the same scenario as having $m$ cars enter the street from right to left with preference list $\alpha'=(n+1-a_1, \dots, n+1-a_m)$ parking under protocol with parameter $1-p$. 
This feature plays a key role in our results for $a_m$ (cf. Remarks \ref{rmk:symmetry} and \ref{rmk:symmetry2}). 

\item In Propositions \ref{cor-asymp} and \ref{cor:mcn}, we specialize to the situation when $m=cn$ for some $0<c<1$ and study the asymptotics. The striking effect of the ``extra spots'' on the rate of convergence of the conditional distribution of $a_m$ to the uniform distribution on $[n]$ is quite evident. Even for small $\varepsilon=1-c$, the presence of an $\varepsilon$ proportion of extra spots decreases the order of the rate of converge to the uniform distribution when $p \neq 1/2$ from $1/\sqrt{n}$ to $1/n$ (cf. \cite[Theorem 6]{DHHRY}). The sharp contrasting behaviors are reflected in Remark \ref{blow} and more visually in Figure \ref{fig:hist}. 
Proposition \ref{prop:c} takes a different perspective and analyzes the convergence rate as a function of $c$, when $n$ and $p$ are treated as fixed parameters. We show that as a function of $c$, the total variation distance between the conditional distribution of $a_m$ and the uniform distribution on $[n]$ decreases to zero at least linearly in $c$, when $c$ goes to zero (cf. Figure \ref{fig:enter-label}).
\end{enumerate}


\section{Key combinatorial and probabilistic results}\label{key}
In this section, we present some combinatorial and probabilistic results that are used in the convergence rate analysis in Section \ref{conv}. Our investigations in this section rely on a combinatorial construction which we term a \emph{parking function multi-shuffle} and \emph{Abel's multinomial theorem}. These concepts were first discussed in Diaconis and Hicks \cite{DH} and later extended in Kenyon and Yin \cite{KY} and further in Yin \cite{Yin}. Some asymptotic expansion formulas also prove useful. We provide some background on these concepts first before diving into our main formulas.

\subsection{Background concepts} \label{sec:background}
Our first result is quite magical. In \cite{DHHRY}, it was established that for the probabilistic one-way parking situation involving $m$ cars and $n$ spots, the probabilities of being a parking function, over the set of all preference lists, add up in a way so that the dependence on $p$ is canceled and there is invariance to the forward probability $p$ for the randomly selected list. For ease of reference, we restate the result below.

\begin{theorem}[Theorem 2 in \cite{DHHRY}]
\label{thm:old}
Consider the preference list $\alpha \in [n]^m$, chosen uniformly at random. Then
\begin{equation}\label{eqn:old}
\PP(\alpha \in \PF{m, n}~|~\alpha \in [n]^m) = \frac{(n-m+1)(n+1)^{m-1}}{n^m}.
\end{equation}
\end{theorem}

We now illustrate the concept of a parking function multi-shuffle through an example. The multi-shuffle construction allows us to compute the number of parking functions $\PF{m, n}$ where the parking preferences of $\ell\leq m$ cars are arbitrarily specified. 
Alternatively, we can think that $\ell$ spots are already taken along a one-way street with $n$ parking spots, and we want to count the possible preferences for the remaining $m-\ell$ cars so that they can all successfully park. 
In the parking function literature, the set of successful preference sequences of the $m-\ell$ cars that enter the street later is referred to as \emph{parking completions} for $\tau=(\tau_1, \dots, \tau_\ell)$ where the entries of $\tau$ denote the $\ell$ spots that are taken previously, arranged in increasing order. This combinatorial construction plays a key role in the proof of Theorem \ref{component}. For more on parking completions we point the interested reader to \cite{bib:ParkingCompletions} and \cite{Yin2}.

\begin{definition}\label{shuffle}
Take $n-m\leq \ell \leq n$ any integer. Let $k=(k_1, \dots, k_\ell) \in [n]^\ell$ be in (strictly) increasing order. 
We say that $a_1, \dots, a_{n-\ell}$ is a \emph{parking function multi-shuffle} of $\ell+1$ parking functions $\alpha_1 \in \PF{k_1-1, k_1-1}, \alpha_2 \in \PF{k_2-k_1-1, k_2-k_1-1}, \dots, \alpha_\ell \in \PF{k_\ell-k_{\ell-1}-1, k_\ell-k_{\ell-1}-1}$, and $\alpha_{\ell+1} \in \PF{n-k_\ell, n-k_\ell}$ if $a_1, \dots, a_{n-\ell}$ is any permutation of the union of the $\ell+1$ words $\alpha_1, \alpha_2+(k_1, \dots, k_1), \dots, \alpha_{\ell+1}+(k_\ell, \dots, k_\ell)$ (some words might be empty), where we allow for permuting the entries in each word.
\end{definition}

\begin{example}
Take $m=8$, $n=10$, $k_1=4$, $k_2=5$, $k_3=6$, and $k_4=8$. Take $\alpha_1=(2, 1, 2) \in \PF{3, 3}$, $\alpha_2=\emptyset$, $\alpha_3=\emptyset$, $\alpha_4=(1) \in \PF{1, 1}$, and $\alpha_5=(2, 1) \in \PF{2, 2}$. Then $(2, \overline{7},2, \underline{9}, \underline{10},1)$ is a multi-shuffle of the five words $(2, 1, 2)$, $\emptyset$, $\emptyset$, $(7)$, and $(10, 9)$. Note that the entries of some of the words are also permuted within.
\end{example}

Next we present a transcription of the famous generalization of the multinomial theorem as introduced by Abel. The following Abel's multinomial theorem  plays an essential role in our theoretical derivations later.

\begin{theorem}[Abel's multinomial theorem, derived from Pitman \cite{Pitman} and Riordan \cite{Riordan}]\label{Abel}
Let
\begin{equation}\label{b}
A_n(x_1, \dots, x_m; p_1, \dots, p_m)=\sum \binom{n}{s} \prod_{j=1}^m (x_j+s_j)^{s_j+p_j},
\end{equation}
where $s=(s_1, \dots, s_m)$ and $\sum_{i=1}^m s_i=n$.
Then
\begin{multline}\label{b1}
A_n(x_1, \dots, x_i, \dots, x_j, \dots, x_m; p_1, \dots, p_i, \dots, p_j, \dots, p_m)\\=A_n(x_1, \dots, x_j, \dots, x_i, \dots, x_m; p_1, \dots, p_j, \dots, p_i, \dots, p_m).
\end{multline}
Also
\begin{multline}\label{b2}
A_n(x_1, \dots, x_m; p_1, \dots, p_m)\\=\sum_{i=1}^m A_{n-1}(x_1, \dots, x_{i-1}, x_i+1, x_{i+1}, \dots, x_m; p_1, \dots, p_{i-1}, p_i+1, p_{i+1}, \dots, p_m).
\end{multline}
Further
\begin{equation}\label{b3}
A_n(x_1, \dots, x_m; p_1, \dots, p_m)=\sum_{s=0}^{n} \binom{n}{s}s!(x_1+s)A_{n-s}(x_1+s, x_2, \dots, x_m; p_1-1, p_2, \dots, p_m).
\end{equation}
Moreover, the following special instances hold via the basic recurrences listed above:
\begin{equation}\label{1}
A_n(x_1, \dots, x_m; -1, \dots, -1)=(x_1\cdots x_m)^{-1}(x_1+\cdots+x_m)(x_1+\cdots+x_m+n)^{n-1}
\end{equation}
and
\begin{equation}\label{2}
A_n(x_1, \dots, x_m; -1, \dots, -1, 0)=(x_1\cdots x_m)^{-1}x_m(x_1+\cdots+x_m+n)^{n}.
\end{equation}
\end{theorem}

Lastly, we introduce a core technical lemma which incorporates ideas from large deviations. This lemma will aid in the asymptotic calculations in Section \ref{sec:asymp}.

\begin{lemma}\label{LD}
Take $0<c<1$. Let $X_1, X_2, \dots$ be iid Poisson$(1)$ random variables. Then
\begin{equation}
\PP(X_1+\cdots+X_n \leq nc)=\frac{\exp\left(-n\left(c\log c-c+1\right)\right)}{\sqrt{2\pi nc}(1-c)}\left(1-\frac{1}{n}\left(\frac{1}{12c}+\frac{c}{(1-c)^2}\right)\right)\left(1+O\left(n^{-2}\right)\right).
\end{equation}
\end{lemma}

\begin{proof}
Let
\begin{equation}
\phi(t)=\mathbb{E} (e^{tX_1})=\exp(e^t-1)
\end{equation}
be the moment generating function of Poisson$(1)$, and $\psi(c, t)=e^{-ct}\phi(t)$. There exists a unique $t^*(c)=\log c$ which minimizes $\psi(c, t)$ with respect to $t$. Write
\begin{equation}
m(c)=\psi(c, t^*(c))=\exp(-c\log c+c-1),
\end{equation}
whose exponent coincides with the negative of the Cram\'{e}r rate function for Poisson$(1)$. We construct auxiliary iid random variables $Y_1, Y_2, \dots$ such that $Y_1$ and $X_1$ have the same range, and for $k=0, 1, \dots$,
\begin{equation}
\PP(Y_1=k)=\frac{1}{m(c)}\PP(X_1=k)\exp(t^*(c)\cdot (k-c))=e^{-c}\frac{c^k}{k!},
\end{equation}
i.e. $Y_1, Y_2, \dots$ are iid Poisson$(c)$ random variables. Using iid-ness, this further gives, for all admissible integers $n$ and $k\leq nc$,
\begin{equation}
\PP(X_1+\cdots+X_n=nc-k)=(m(c))^n e^{kt^*(c)}\PP(Y_1+\cdots+Y_n=nc-k).
\end{equation}
Recall that an integer-valued random variable $U$ with characteristic function $\eta$ may be alternatively expressed as
\begin{equation}
\PP(U=u)=\frac{1}{2\pi} \int_{-\pi}^\pi e^{-itu} \eta(t) dt.
\end{equation}
Therefore, as $n\rightarrow \infty$, we have that the
\begin{equation}
\PP(X_1+\cdots+X_n \leq nc) \sim \frac{(m(c))^n}{2\pi} \int_{-\pi}^\pi \frac{1}{1-ze^{it}} \psi^n(t) dt,
\end{equation}
where $z=e^{t^*(c)}=c$ and $\psi(t)$ is the characteristic function of $Y_1-c$.

We compute the central moments of $Y_1$:
\begin{equation}
\mu_2=c, \hspace{.2cm} \mu_3=c, \hspace{.2cm} \mu_4=3c^2+c.
\end{equation}
It follows that
\begin{align}
\frac{\mu_4}{\mu_2^2}-3-\frac{5\mu_3^2}{3\mu_2^3}&=-\frac{2}{3c}
\intertext{and}
\frac{-z\frac{\mu_3}{\mu_2}+z\frac{1+z}{1-z}}{(1-z)\mu_2}&=\frac{2c}{(1-c)^2}.
\end{align}
The conclusion then follows from applying precise large deviation asymptotic estimation as in Bahadur and Rao \cite{BR} and Blackwell and Hodges \cite{BH}, but adapted to the left tail. See also Dembo and Zeitouni \cite[Theorem 3.7.4]{DZ} for a summary statement.
\end{proof}

\begin{remark}
The asymptotic precision in the above Lemma \ref{LD} may be sharpened if we keep more terms in the expansion.
\end{remark}

\subsection{Main formulas} \label{sec:asymp}
We are now ready to provide explicit formulas which show how the parking statistic $a_m$, the parking preference of the last car, depends on $p$. 
In Theorem \ref{component} we give the distribution of $a_m$ and in Theorem \ref{mean} we calculate its expected value. 
Both formulas depict explicit $p$ dependence. 
Theorems \ref{component} and \ref{mean} are extensions of the corresponding results in \cite{DHHRY}, where the probabilistic scenario of parking $n$ cars on a one-way street with $n$ spots are investigated. 
Parking $m \leq n$ cars instead of $n$ cars on $n$ spots opens up many new possibilities, and as such, the current results and their proofs are more complicated than the special situation explored in \cite{DHHRY}. After proving these initial results, we then specialize to the situation when $m=cn$ for some $0<c<1$ and study the asymptotics in Proposition \ref{cor-asymp}. In particular, we establish some sharp contrasting behavior (cf. Remark \ref{blow}).

\begin{theorem}\label{component}
Consider the preference list $\alpha=(a_1,a_2,\ldots,a_m) \in [n]^m$, chosen uniformly at random. Then given that $\alpha \in \PF{m, n}$, we have
\begin{align}\label{prob-long-equation}
&\PP(a_m=j~|~\alpha \in \PF{m, n})=\frac{n-m+2}{(n-m+1)(n+1)}-\frac{1}{(n+1)^{m-1}} \cdot \notag \\
&\cdot \Big[p\sum_{s=n-j+1}^{m-1} \binom{m-1}{s} (n-s)^{m-s-2} (s+1)^{s-1} +(1-p)\sum_{s=j}^{m-1} \binom{m-1}{s} (n-s)^{m-s-2} (s+1)^{s-1}\Big],
\end{align}
where $a_m$ denotes the parking preference of the last car.
\end{theorem}

\begin{remark}
In \eqref{prob-long-equation}, when the lower index of summation is larger than the upper index we interpret the sum as zero by convention.
\end{remark}

\begin{remark}\label{rmk:symmetry}
Note the parking symmetry as observed in the introduction: $\PP(a_m=j|\alpha \in \PF{m, n})$ under protocol with parameter $p$ equals $\PP(a_m=n+1-j|\alpha \in \PF{m, n})$ under protocol with parameter $1-p$.
\end{remark}

\begin{proof}[Proof of Theorem~\ref{component}]
Before the $m$th car enters, cars $1, \dots, m-1$ have all parked along the one-way street with $n$ spots, leaving $n-m+1$ open spots for the $m$th car to park in. Let $k_i$ for $i=1, \dots, n-m+1$ represent these spots, so that $0\eqqcolon k_0<k_1<\cdots<k_{n-m+1}<k_{n-m+2}\coloneqq n+1$. Since a car cannot jump over an empty spot, the parking protocol implies that $(a_1, \dots, a_{m-1})$ corresponds to a parking function multi-shuffle of $n-m+2$ parking functions. In other words, $(a_1, \dots, a_{m-1})$ may be decomposed into $n-m+2$ disjoint non-interacting segments (some segments might be empty), with each segment a parking function of length $(k_{i}-k_{i-1}-1)$ after translation. The open spot $k_i$ where $i=1, \dots, n-m+1$ could be either the same as $j$, the preference of the last car, in which case the car parks directly. Or, $k_i$ could be bigger than or less than $j$, in which case where the last car parks depends on the outcome of the biased coin flip as it will dictate the car to go forward or backward. Note that when $j<k_1$ (resp. $j>k_{n-m+1}$) only the forward (resp. backward) movement of the last car will result in a successful parking scenario, whereas for all other $j$'s, the last car will always be able to park as there are open spots both behind and ahead of this car. Using Theorem \ref{thm:old}, we have
\begin{align}\label{long}
\PP(a_m=j&~|~\alpha \in \PF{m, n})=\frac{1}{(n-m+1)(n+1)^{m-1}} \cdot \notag \\
&\cdot \Big[\sum_{k_1, \dots, k_{n-m+1}} \binom{m-1}{k_1-k_0-1, \dots, k_{n-m+2}-k_{n-m+1}-1} \prod_{i=1}^{n-m+2} (k_{i}-k_{i-1})^{k_{i}-k_{i-1}-2} \notag \\
&\hspace{2cm} -p\sum_{k=n-m+1}^{j-1} (n-m+1) \binom{m-1}{n-k} k^{m-n+k-2} (n-k+1)^{n-k-1} \notag \\
&\hspace{2cm} -(1-p)\sum_{k=j+1}^{m} (n-m+1)\binom{m-1}{k-1} k^{k-2} (n-k+1)^{m-k-1}\Big].
\end{align}
The first summation accounts for all possible locations of the empty spots $k_1, \dots, k_{n-m+1}$ and all possible movement of the last car (assuming that it always parks). 
The middle summation accounts for $j>k_{n-m+1}$ and the last car moves forward (thus failing to park) and the last summation accounts for $j<k_1$ and the last car moves backward (thus failing to park). 
We also note that implicitly the first empty spot $k_1\leq m$ and the last empty spot $k_{n-m+1}\geq n-m+1$. 
For the last empty spot $k_{n-m+1}$, there are $m-n+k_{n-m+1}-1$ cars and $k_{n-m+1}-1$ spots to its left and $n-k_{n-m+1}$ cars and $n-k_{n-m+1}$ spots to its right. 
Similarly, for the first empty spot $k_1$, there are $k_1-1$ cars and $k_1-1$ spots to its left and $m-k_1$ cars and $n-k_1$ spots to its right. The multinomial coefficients come from the parking function multi-shuffle construction and the other multiplicative factors come from \eqref{eqn:old}.

To simplify \eqref{long}, we set $s=(k_1-k_0-1, \dots, k_{n-m+2}-k_{n-m+1}-1)$ with $\sum_{i=1}^{n-m+2} s_i=m-1$. By Abel's multinomial identity \eqref{1}, the first summation reduces to $$A_{m-1}(1, \dots, 1; -1, \dots, -1)=(n-m+2)(n+1)^{m-2}.$$ We further set $s=n-k$ in the middle summation and $s=k-1$ in the last summation. Together this gives
\begin{align*}
&\PP(a_m=j~|~\alpha \in \PF{m, n})=\frac{n-m+2}{(n-m+1)(n+1)}-\frac{1}{(n+1)^{m-1}} \cdot \notag \\
&\cdot \Big[p\sum_{s=n-j+1}^{m-1} \binom{m-1}{s} (n-s)^{m-s-2} (s+1)^{s-1} +(1-p)\sum_{s=j}^{m-1} \binom{m-1}{s} (n-s)^{m-s-2} (s+1)^{s-1}\Big].\hfill\qedhere
\end{align*}
\end{proof}
We now give the expected value for the preference of the last car.
\begin{theorem}\label{mean}
For preference list $\alpha \in [n]^m$ chosen uniformly at random, we have
\begin{equation}\label{identity}
\mathbb{E}(a_m~|~\alpha \in \PF{m, n})= 
       \frac{n+2p}{2}-\left(p-\frac{1}{2}\right)\frac{e^{n+1}(m-1)! \, \PP(Z\le m-1)}{(n+1)^{m-1}},
\end{equation}
where $Z$ is a Poisson random variable with parameter $\lambda = n+1$. Alternatively, the  identity in \eqref{identity} can be written as
\begin{equation*}
\mathbb{E}(a_m~|~\alpha \in \PF{m, n})= 
       \frac{n+2p}{2}-\left(p-\frac{1}{2}\right)\frac{e^{n+1}\Gamma(m,n+1)}{(n+1)^{m-1}},
\end{equation*}
where $\Gamma(s,x)$ is the upper incomplete gamma function
$$
\Gamma(s,x) = \int_x^{\infty}t^{s-1}e^{-t}{\rm d}t.
$$
\end{theorem}

\begin{remark}\label{rmk:symmetry2}
From the parking symmetry described in the introduction, the sum of $\mathbb{E}(a_m~|~\alpha \in \PF{m, n})$ under protocol with parameter $p$ and $\mathbb{E}(a_m~|~\alpha \in \PF{m, n})$ under protocol with parameter $1-p$ is $n+1$.
\end{remark}

\begin{proof}[Proof of Theorem~\ref{mean}]
By Theorem~\ref{component},
\begin{align}
&\mathbb{E}(a_m~|~\alpha \in \PF{m, n})=\sum_{j=1}^n j\PP(a_m=j~|~\alpha \in \PF{m, n}) \notag \\
&=\frac{n(n-m+2)}{2(n-m+1)}-\frac{1}{(n+1)^{m-1}} \sum_{j=1}^n j\Big[p\sum_{s=n-j+1}^{m-1} \binom{m-1}{s} (n-s)^{m-s-2} (s+1)^{s-1}\notag \\
&\hspace{5cm}+(1-p) \sum_{s=j}^{m-1} \binom{m-1}{s} (n-s)^{m-s-2} (s+1)^{s-1}\Big] \notag \\
&=\frac{n(n-m+2)}{2(n-m+1)}-\frac{1}{(n+1)^{m-1}} \sum_{s=0}^{m-1} \binom{m-1}{s} (n-s)^{m-s-2} (s+1)^{s-1} \Big[p\sum_{j=n-s+1}^n j+(1-p)\sum_{j=1}^{s} j \Big] \notag \\
&=\frac{n(n-m+2)}{2(n-m+1)}-\frac{1}{2(n+1)^{m-1}} \Bigg[\sum_{s=0}^{m-1} \binom{m-1}{s} (n-s)^{m-s-2} (s+1)^{s-1} \cdot \notag \\
&\hspace{2cm} \cdot [(2p-1)(n-s)(s+1)+(n+2p)(s+1)-2p(n+1)]\Bigg] \label{A1} \\
&=\frac{n(n-m+2)}{2(n-m+1)}-\frac{1}{2(n+1)^{m-1}} \Big[(2p-1)A_{m-1}(n-m+1, 1; 0, 0)\notag \\
&\hspace{2cm}+(n+2p)A_{m-1}(n-m+1, 1; -1, 0)-2p(n+1)A_{m-1}(n-m+1, 1; -1, -1) \Big], \label{A2}
\end{align}
where Abel's binomial theorem is used multiple times from \eqref{A1} to \eqref{A2}.

Using \eqref{1} yields
\begin{equation}\label{eq:Am-1-1}
    A_{m-1}(n-m+1, 1; -1, -1)=\frac{n-m+2}{n-m+1}(n+1)^{m-2},
\end{equation} 
and using \eqref{2} yields
\begin{equation}\label{eq:Am-10}
    A_{m-1}(n-m+1, 1; -1, 0)=\frac{1}{n-m+1}(n+1)^{m-1}.
\end{equation}
Further note that by \eqref{b3} and \eqref{2}, we have that 
\begin{align}\label{eq:Am00}
A_{m-1}(n-m+1, 1; 0, 0)=\sum_{s=0}^{m-1} \binom{m-1}{s} (n+1)^s (m-s-1)!=(m-1)! \sum_{s=0}^{m-1} \frac{(n+1)^s}{s!}.
\end{align}
We recognize that
\begin{equation}\label{eq:Am00Poi}
e^{-(n+1)} \sum_{s=0}^{m-1} \frac{(n+1)^s}{s!}
\end{equation}
equals the probability that a Poisson random variable  $Z$ with parameter $\lambda=n+1$ is less than or equal to $m-1$. 
Substituting \eqref{eq:Am00Poi} into \eqref{eq:Am00} gives
\begin{equation}\label{substitute}
A_{m-1}(n-m+1, 1; 0, 0) = e^{n+1}(m-1)! \, \PP(Z \le m-1).
\end{equation}
Substituting \eqref{eq:Am-1-1}, \eqref{eq:Am-10}, and \eqref{substitute} back into \eqref{A2} yields
\begin{equation}\label{middle-step}
    \begin{split}
        \mathbb{E}(a_m~|~\alpha \in \PF{m, n}) & = \frac{n(n-m+2)}{2(n-m+1)} -\frac{(2p-1)e^{n+1}(m-1)! \, \PP(Z\le m-1)}{2(n+1)^{m-1}}\\
        & \quad \quad \quad -\frac{n+2p}{2(n-m+1)} + \frac{p(n-m+2)}{n-m+1}.
    \end{split}
\end{equation}
To finish the proof, observe that in \eqref{middle-step}, we have that
\begin{equation*}
        \frac{n(n-m+2)}{2(n-m+1)} -\frac{n+2p}{2(n-m+1)} + \frac{p(n-m+2)}{n-m+1}=\frac{n+2p}{2}.\qedhere
\end{equation*}

\end{proof}
We now provide an asymptotic result for the expected value of the preference of the last car.
\begin{proposition}\label{cor-asymp}
Take $m$ and $n$ large with $m=cn$ for some $0<c<1$. For preference list $\alpha \in [n]^m$ chosen uniformly at random, we have
\begin{equation}\label{eq:correction}
\mathbb{E}(a_m~|~\alpha \in \PF{m, n})=\frac{n+1}{2}-(2p-1) \frac{c}{2(1-c)}-\frac{1}{n} (2p-1) \frac{c^2-c-1}{2(1-c)^3}+O\left(n^{-2}\right).
\end{equation}
\end{proposition}

\begin{remark}
The lower order correction terms from $(n+1)/2$ for $\mathbb{E}(a_m~|~\alpha \in \PF{m, n})$ vanish completely under protocol with parameter $p=1/2$, and $\mathbb{E}(a_m~|~\alpha \in \PF{m, n})=(n+1)/2$ exactly.
\end{remark}

\begin{remark}\label{blow}
Recall that in \cite[Theorem 4]{DHHRY}, it was derived that for $m=n$,
\begin{equation}\label{eq:special}
\mathbb{E}(a_n~|~\alpha \in \PF{n, n})=\frac{n+1}{2}-(2p-1)\Big[\frac{\sqrt{2\pi}}{4}n^{1/2}-\frac{7}{6}\Big]+o(1).
\end{equation}
As $c\rightarrow 1$, the correction terms in \eqref{eq:correction} blow up, contributing to the different asymptotic orders between the generic situation $m=cn$ with $0<c<1$,  described by \eqref{eq:correction}, and the special situation $m=n$, described by \eqref{eq:special}.
\end{remark}

\begin{proof}[Proof of Proposition \ref{cor-asymp}]
    The proof follows from Theorem \ref{mean} and a careful asymptotic analysis of the Poisson term \eqref{substitute}. We provide two approaches that offer different perspectives. 
    

\vskip.1truein

\noindent \underline{First approach.} 
Notice that from Stirling's formula,
\begin{align}
(m-1)! \sim \sqrt{2\pi (m-1)}e^{-(m-1)} (m-1)^{m-1} \left[1+\frac{1}{12(m-1)}\right].
\end{align}
We also recall that
\begin{equation}
e^{-(n+1)} \sum_{s=0}^{m-1} \frac{(n+1)^s}{s!}
\end{equation}
equals the probability that the sum of $n+1$ iid Poisson$(1)$ random variables is less than or equal to $m-1$, and is asymptotic to
\begin{align}
&e^{-(n+1)} e^{(m-1)}\frac{\left(\frac{m-1}{n+1}\right)^{-(m-1)}}{\sqrt{2\pi (m-1)}\left(1-\frac{m-1}{n+1}\right)} \left(1-\frac{1}{n}\left(\frac{1}{12c}+\frac{c}{(1-c)^2}\right)+O\left(n^{-2}\right)\right) \notag \\
&=e^{-(n+1)} e^{(m-1)}\frac{\left(\frac{m-1}{n+1}\right)^{-(m-1)}}{\sqrt{2\pi (m-1)}\left(1-c\right)} \left(1-\frac{1}{n}\left(\frac{1}{12c}+\frac{c}{(1-c)^2}+\frac{c+1}{1-c}\right)+O\left(n^{-2}\right)\right)
\end{align}
from the large deviation expansion in Lemma~\ref{LD}. Dividing by $(n+1)^{m-1}$ and simplifying we get
\begin{equation}\label{cf}
\frac{1}{(n+1)^{m-1}}A_{m-1}(n-m+1, 1; 0, 0)=\frac{1}{1-c} \left(1+\frac{1}{n} \frac{c^2-c-1}{(1-c)^2}+O\left(n^{-2}\right)\right).
\end{equation}

\vskip.1truein

\noindent \underline{Second approach.} We write
\begin{align}\label{eq:before}
A_{m-1}(n-m+1, 1; 0, 0)&=\sum_{s=0}^{m-1} \binom{m-1}{s} (n-s)^{m-s-1} (s+1)^{s} \notag \\
&=n^{m-1} \sum_{s=0}^{m-1} \frac{(ce^{-c})^s}{s!} (s+1)^s \left(1-\frac{s(s+1)}{2cn}+\frac{s(s+1)}{n}-\frac{s^2 c}{2n}+O\left(n^{-2}\right)\right).
\end{align}

The tree function $F(z) = \sum_{s=0}^\infty \frac{z^s}{s!}(s+1)^{s-1}$ is related to the Lambert $W$ function via $F(z)=-W(-z)/z$, and
satisfies $F(ce^{-c}) = e^c$. The function $G(z)=\sum_{s=0}^\infty \frac{z^s}{s!} (s+1)^{s}$ is further related to the tree function $F(z)$ via $G(z)=zF'(z)+F(z)$. By the chain rule, $G(z)$ and its first and second derivatives respectively satisfy
\begin{align*}
G(ce^{-c}) = \frac{e^{c}}{1-c}, \quad G'(ce^{-c}) = \frac{2-c}{(1-c)^3}e^{2c}, \quad\mbox{and}\quad
G''(ce^{-c}) = \frac{2c^2-8c+9}{(1-c)^5} e^{3c}.
\end{align*}
We recognize that \eqref{eq:before} converges to
\begin{equation}
n^{m-1}\left(\sum_{s=0}^\infty\frac{(ce^{-c})^s}{s!}(s+1)^{s}\left(1+\frac1n(As+Bs^2)+O\left(n^{-2}\right)\right)\right),\label{lamb}
\end{equation}
where $A=-\frac{1}{2c}+1$ and $B=-\frac{1}{2c}+1-\frac{c}{2}.$
Using $G(z)$ with $z=ce^{-c}$, \eqref{lamb} can be written as 
\begin{equation}
n^{m-1}\left(G(z) + \frac1n \Big(AzG'(z) + B(z^2G''(z)+zG'(z))\Big)+O\left(n^{-2}\right)\right).
\end{equation}
Dividing \eqref{eq:before} by $(n+1)^{m-1}$ and simplifying we get
\begin{equation}
\frac{1}{(n+1)^{m-1}}A_{m-1}(n-m+1, 1; 0, 0)=\frac{1}{1-c} \left(1+\frac{1}{n} \frac{c^2-c-1}{(1-c)^2}+O\left(n^{-2}\right)\right).
\end{equation}


Combining the above, \eqref{A2} becomes
\begin{align*}
&\mathbb{E}(a_m~|~\alpha \in \PF{m, n})=\frac{n+2p}{2}-\left(p-\frac{1}{2}\right) \frac{1}{1-c} \left(1+\frac{1}{n} \frac{c^2-c-1}{(1-c)^2}+O\left(n^{-2}\right)\right) \notag \\
&=\frac{n+1}{2}-(2p-1) \frac{c}{2(1-c)}-\frac{1}{n} (2p-1) \frac{c^2-c-1}{2(1-c)^3}+O\left(n^{-2}\right).\qedhere
\end{align*}
\end{proof}

\section{Convergence rates}\label{conv}
In this section, we measure how close the conditional distribution of $a_m$ (investigated in Theorem \ref{mean} and Proposition \ref{cor-asymp}) is to the uniform distribution on $[n]$, which we denote by $\unif{n}$. 
Throughout this section, the notion of distance used  is the {\it total variation distance}. 
We recall that for two probability distributions $P$ and $Q$ over $[n]$, their total variation distance (TV) is given by
\begin{equation}\label{def:dTV}
    \| P - Q \|_{\TV} \coloneqq  \frac{1}{2}\sum_{j=1}^n| P(j) - Q(j) |.
\end{equation}
To simplify the notation, we write 
\begin{equation}\label{def:Qmnp}
    Q_{m,n,p} (\; \cdot \;) = \PP(a_m= \; \cdot \; | \; \alpha \in \PF{m, n}),
\end{equation}
where the conditional probability $\PP$ is under parking protocol with parameter $p$.
We begin by establishing a result that is useful for handling the total variation distance between $Q_{m,n,p}$ and ${\rm Uni}_n$. 
In words, Proposition \ref{prop:generalboundstv} tells us that establishing bounds for $\|Q_{m,n,p} - {\rm Uni}_n \|_{\TV}$ can be reduced to the classical case $p=1$.
\begin{proposition}\label{prop:generalboundstv} 
    For all $m,n \in \mathbb{N}$ with $m\le n$ and $p \in [0,1]$, the following bounds hold:
    $$
      |2p-1| \|Q_{m,n,1} - {\rm Uni}_n \|_{\TV}\le  \|Q_{m,n,p} - {\rm Uni}_n \|_{\TV} \le  \|Q_{m,n,1} - {\rm Uni}_n \|_{\TV}.
    $$
\end{proposition}
\begin{proof}
    Theorem \ref{component} implies that for any $m, n \in \mathbb{N}$ and $p \in [0,1]$, $Q_{m,n,p}$ is a convex combination of $Q_{m,n,1}$ and $Q_{m,n,0}$, so 
    $$
        Q_{m,n,p} = (1-p)Q_{m,n,0} + pQ_{m,n,1}.
    $$
    Also, $Q_{m,n,1}(j)=Q_{m,n,0}(n+1-j)$ for any $j\in [n]$ (see Remark \ref{rmk:symmetry}).
    The proof then parallels that of \cite[Proposition 24]{DHHRY}.
\end{proof}
In order to identify the lower bound for the total variation distance, the following alternative version of the total variation distance is instrumental.
\begin{proposition}[Proposition 4.5 in Levin and Peres \cite{levin2017markov}]\label{prop:altdTV} Let $P$ and $Q$ be two probability distributions over $[n]$, then 
    \begin{equation}\label{test}
        \| P - Q \|_{\TV} = \frac{1}{2} \sup \left \lbrace \sum_{j=1}^n f(j)P(j) - \sum_{j=1}^nf(j)Q(j): \max_j|f(j)| \le 1 \right \rbrace.
    \end{equation}
\end{proposition}
We are now ready to state the main theorem of this section.
\begin{theorem}\label{thm:tvd}For any $m \leq n$ and $p \in [0,1]$, the following inequalities hold:
$$
\|Q_{m,n,p} - {\rm Uni}_n\|_{\TV} \le \frac{m-1}{(n+1)(n-m+1)},
$$
and
    $$
        \|Q_{m,n,p} - {\rm Uni}_n\|_{\TV} \ge \frac{|2p-1|}{4n} \left| 1 - \frac{e^{n+1}(m-1)!\, \PP(Z\le m-1)}{(n+1)^{m-1}} \right|,
    $$
    where $Z$ obeys a Poisson distribution with parameter $\lambda = n+1$.
 Additionally, for $p=1/2$, we have
$$
\|Q_{m,n,1/2} - {\rm Uni}_n\|_{\TV} \ge \frac{1}{2}\left| \frac{n^{m-2}}{2(n+1)^{m-1}} - \frac{1}{2(n+1)}\left[1 + \frac{n-2m+2}{n(n-m+1)}\right] \right|.
$$
\end{theorem}
Before presenting the proof of Theorem \ref{thm:tvd}, we discuss an important and interesting consequence of the theorem, which we summarize in the following result.
\begin{proposition}\label{cor:mcn} Take $m=cn$ for some $0<c<1$. Then for all $p \in [0, 1]$,
    $$
    \|Q_{m,n,p} - {\rm Uni}_n\|_{\TV} = \Theta\left( \frac{1}{n}\right).
    $$
\end{proposition}
Proposition \ref{cor:mcn} depicts the striking effect of the ``extra spots'' on the rate of convergence of $Q_{m, n, p}$ to the uniform distribution ${\rm Uni}_n$. In \cite[Theorem 6]{DHHRY}, the authors establish that when $m=n$ and $p\neq 1/2$, the correct order of the rate of convergence is~$1/\sqrt{n}$, while in Proposition \ref{cor:mcn}, this rate of convergence is shown to be $1/n$. 
Our result thus illustrates the strong impact of having fewer cars than spots ($m<n$) on the preference distribution. Even for small $\varepsilon=1-c$, the presence of an $\varepsilon$ proportion of extra spots changes the order of the rate of converge to the uniform distribution completely when $p \neq 1/2$. 
The histograms in Figure \ref{fig:hist} provide an illustration of this discussion.

Notice that in the case of $m<n$, the parking preference distribution of the last car (left plot in Figure \ref{fig:hist}) is quite close to the uniform distribution over $[n]$. Whereas, in the case of $m=n$, more mass is placed at the head of the distribution (right plot in Figure \ref{fig:hist}). This is expected since, as explained above, the convergence of $Q_{m, n, p}$ to the uniform distribution ${\rm Uni}_n$ in the generic situation $p\neq 1/2$ is faster when there are more spots than cars (convergence rate decreases from $1/n$ to $1/\sqrt{n}$).

\begin{figure}[h!]
  \centering
  \begin{subfigure}{0.45\textwidth}
    \centering
    \includegraphics[width=\textwidth]{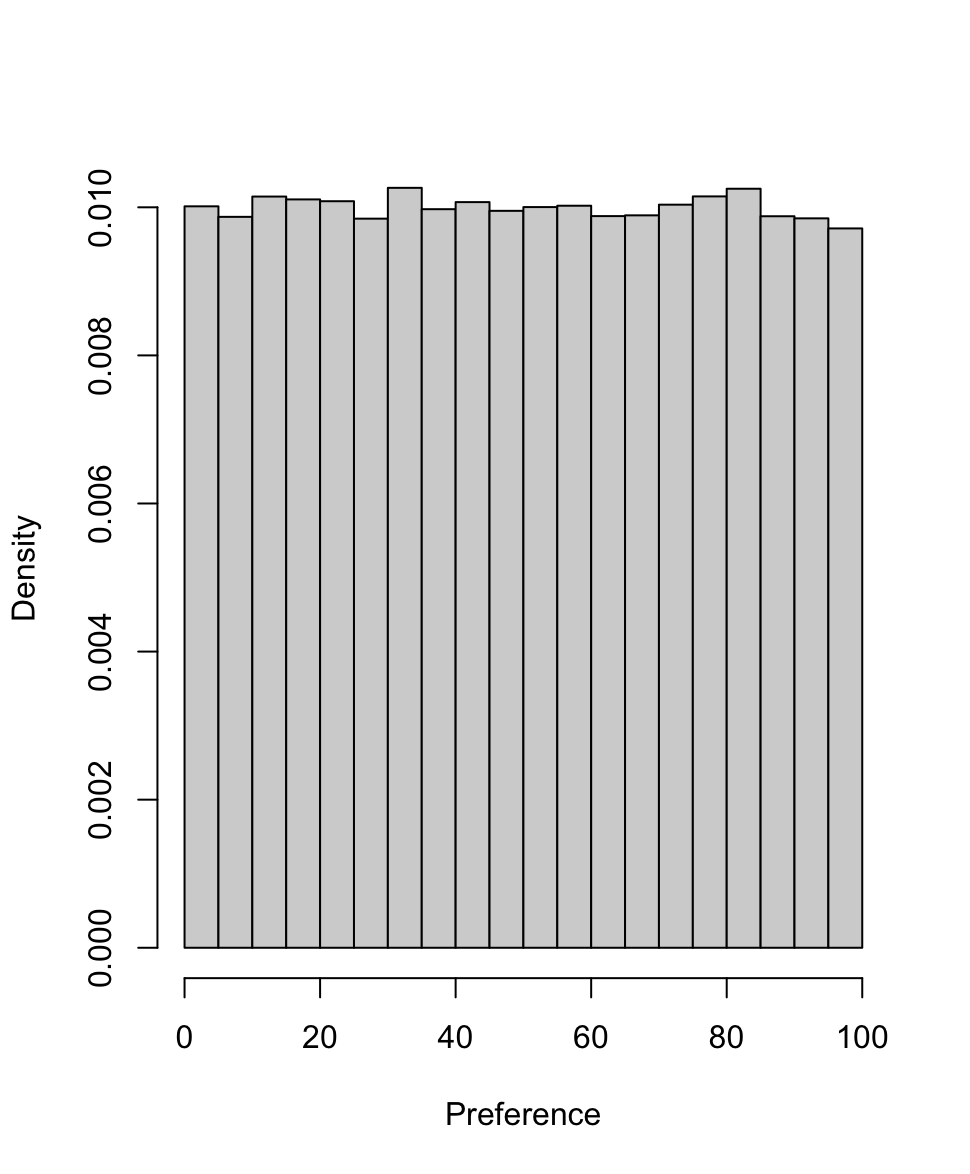}
    \caption{$n=100, m=20$ and $p=0.7$}
    \label{fig:figure1}
  \end{subfigure}
  \hfill
  \begin{subfigure}{0.45\textwidth}
    \centering
    \includegraphics[width=\textwidth]{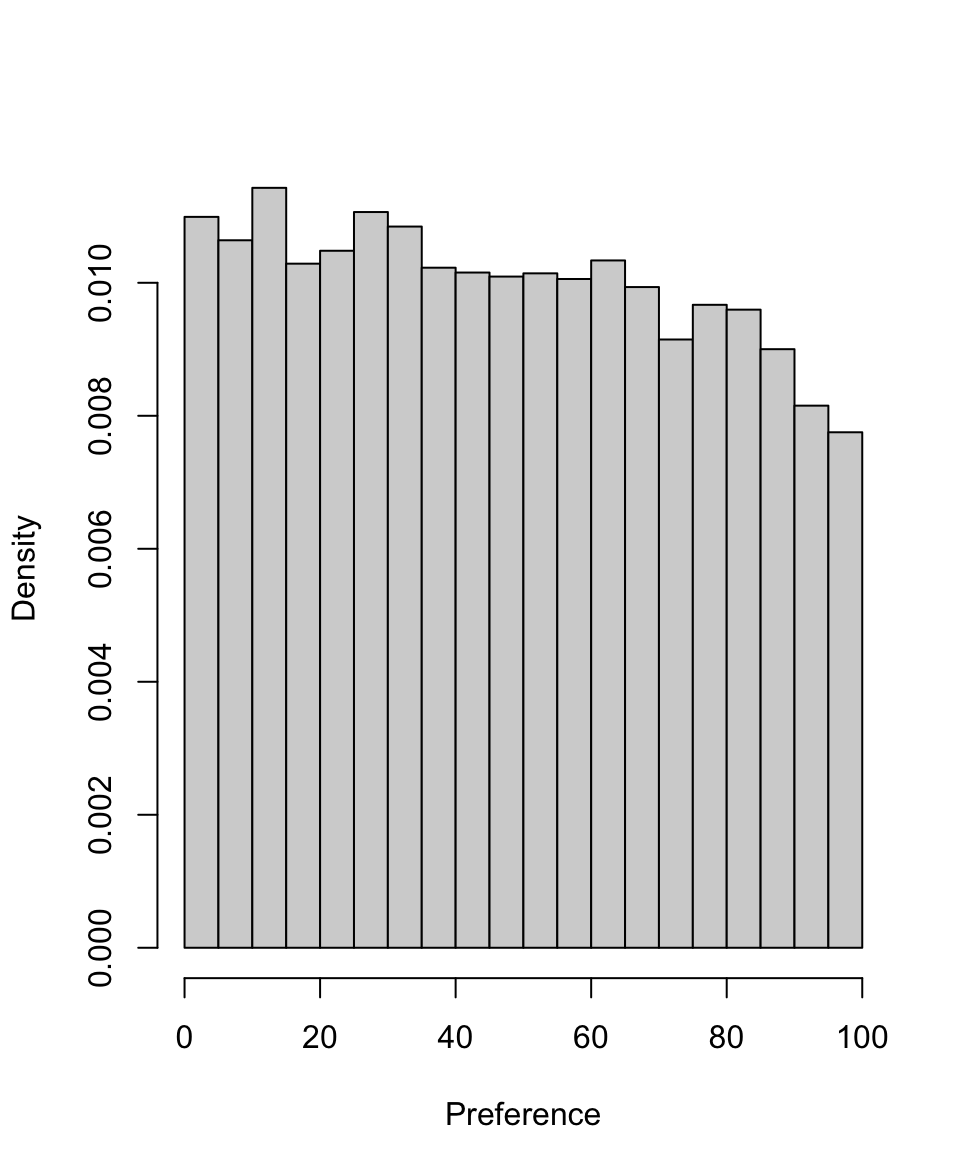}
    \caption{$n = m = 100$ and $p=0.7$}
    \label{fig:figure2}
  \end{subfigure}
  \caption{The conditional distribution $Q_{m,n,p}$ of $a_m$ (parking preference of the last car) in $100,000$ samples of preference lists of size $n=100$ chosen uniformly at random. Here the forward-moving probability $p=0.7$. The left plot is for $m=20$ cars and the right plot is for $m=100$ cars.}
  \label{fig:hist}
\end{figure}

We proceed to show how Proposition \ref{cor:mcn} follows from Theorem \ref{thm:tvd}.
\begin{proof}[Proof of Proposition \ref{cor:mcn}] Notice that  by Theorem \ref{thm:tvd}, with $m = cn$, we have that for all $p\in [0,1]$,
\begin{equation}\label{eq:up}
    \|Q_{m,n,p} - {\rm Uni}_n\|_{\TV} \le \frac{cn-1}{(n+1)[(1-c)n+1]} \sim \frac{c}{(1-c)n}.
\end{equation}
This allows us to conclude that our upper bound for the total variation distance is asymptotic to $c/\left((1-c)n\right)$.

    For the lower bound, for $p\neq 1/2$, Theorem \ref{thm:tvd} gives
    \begin{equation}
        \|Q_{m,n,p} - {\rm Uni}_n\|_{\TV}  \ge \frac{|2p-1|}{4n} \left| 1 - \frac{e^{n+1}(cn-1)! \, \PP(Z\le cn-1)}{(n+1)^{cn-1}} \right|,
    \end{equation}
where $Z$ is a Poisson random variable with parameter $\lambda = n+1$, or equivalently, $Z$ is the sum of $n+1$ independent Poisson random variables with parameter $\lambda =1$. Thus, by the first approach in the proof of Proposition \ref{cor-asymp} (cf. in particular \eqref{cf}), we have  that
\begin{equation}
 \|Q_{m,n,p} - {\rm Uni}_n\|_{\TV} \gtrsim \frac{|2p-1|}{4n} \left| 1-\frac{1}{1-c} \right|,
\end{equation}
which guarantees a lower bound asymptotic to $\left(|2p-1|c\right)/\left(4(1-c)n\right)$ for the total variation distance when $p\neq 1/2$. 

Lastly, for the case $p=1/2$, we use the second part of Theorem \ref{thm:tvd}, which gives that 
\begin{equation}
    \begin{split}
        \|Q_{m,n,1/2} - {\rm Uni}_n\|_{\TV}  & \ge \frac{1}{2}\left| \frac{n^{cn-2}}{2(n+1)^{cn-1}} - \frac{1}{2(n+1)}\left[1 + \frac{(1-2c)n+2}{n[(1-c)n+1]}\right] \right|.
    \end{split}
\end{equation}
Notice that $n^{cn-2}/\left(2(n+1)^{cn-1}\right) \sim e^{-c}/(2n)$, whereas $\left((1-2c)n+2\right)/\left(n((1-c)n+1)\right) \sim (1-2c)/\left((1-c)n\right)$, which combined ensures a lower bound of the desired asymptotic order for the total variation distance when $p=1/2$. 
\end{proof}

We now present the proof of Theorem \ref{thm:tvd}.

\begin{proof}[Proof of Theorem \ref{thm:tvd}]
We establish the lower bound first and then the upper bound.\\

\noindent \underline{Lower bound, $p\neq 1/2$.} By Propositions \ref{prop:generalboundstv} and \ref{prop:altdTV}, it is enough to find a suitable lower bound for the case $p=1$ using test functions. Let $f$ be a function over $[n]$ defined as $f(j) = j/n$. Then by Proposition \ref{prop:altdTV},
$$
    \|Q_{m,n,1} - {\rm Uni}_n\|_{\TV}  \ge \frac{1}{2}\sum_{j=1}^n\left[\frac{j}{n}Q_{m,n,1}(j) - \frac{j}{n^2}\right] =  \frac{\mathbb{E}(a_m~|~\alpha \in \PF{m, n})}{2n} - \frac{n+1}{4n} ,
$$
where the conditional expectation $\mathbb{E}$ is under parking protocol with $p=1$. 
By taking $g(j)= -j/n$ and applying Proposition \ref{prop:altdTV} again, we further conclude that
\begin{align}
    \|Q_{m,n,1} - {\rm Uni}_n\|_{\TV} & \ge  \left|\frac{n+1}{4n} - \frac{\mathbb{E}(a_m~|~\alpha \in \PF{m, n})}{2n} \right|\label{eq:lbdtv}.
\end{align}
By Theorem \ref{mean} with $p=1$, we have that 
\begin{equation}\label{eq-poisson}
    \frac{\mathbb{E}(a_m~|~\alpha \in \PF{m, n})}{2n} = \frac{n+2}{4n}-\frac{e^{n+1}(m-1)! \, \PP(Z\le m-1)}{4n(n+1)^{m-1}},
\end{equation}
where $Z$ is a Poisson random variable with $\lambda = n+1$. 
Substituting \eqref{eq-poisson} into \eqref{eq:lbdtv}, we obtain
\begin{equation*}
    \begin{split}
        \|Q_{m,n,1} - {\rm Uni}_n\|_{\TV} & \ge \left| \frac{1}{4}+\frac{1}{4n}  - \frac{1}{4} - \frac{1}{2n} + \frac{e^{n+1}(m-1)! \, \PP(Z\le m-1)}{4n(n+1)^{m-1}} \right| \\
        & = \left| \frac{e^{n+1}(m-1)! \, \PP(Z\le m-1)}{4n(n+1)^{m-1}} - \frac{1}{4n}\right|.
    \end{split}
\end{equation*}
Finally, Proposition \ref{prop:generalboundstv} implies that 
$$
\|Q_{m,n,p} - {\rm Uni}_n\|_{\TV} \ge |2p-1|\|Q_{m,n,1} - {\rm Uni}_n\|_{\TV},
$$
which shows the desired lower bound. 
\\

\noindent \underline{Lower bound, $p=1/2$.} Notice that the previous lower bound is $0$ when $p=1/2$. For this reason, we handle the case $p=1/2$ separately. We adopt a direct approach. By the definition of the total variation distance given in \eqref{def:dTV}, we have that
\begin{equation}\label{eq:lb1}
    \|Q_{m,n,1/2} - {\rm Uni}_n\|_{\TV} \ge \frac{1}{2}\left|Q_{m,n,1/2}(1) - \frac{1}{n}\right|.
\end{equation}
On the other hand, by Theorem \ref{component} with $j=1$, we have that
\begin{align}
        Q_{m,n,1/2}(1) &  =\frac{n-m+2}{(n-m+1)(n+1)}-\frac{1}{2(n+1)^{m-1}} \cdot \sum_{s=1}^{m-1} \binom{m-1}{s} (n-s)^{m-s-2} (s+1)^{s-1}  \notag\\
        & = \frac{n-m+2}{(n-m+1)(n+1)}-\frac{A_{m-1}(n-m+1,1;-1,-1)}{2(n+1)^{m-1}} + \frac{n^{m-2}}{2(n+1)^{m-1}},\label{here}
    \end{align}
where in the second equality we add and subtract $\frac{1}{2(n+1)^{m-1}} \binom{m-1}{0}n^{m-2} 1^{-1}$ on the right-hand side. Now recall that, by \eqref{1},
\begin{align}\label{subin}
A_{m-1}(n-m+1,1;-1,-1) = \frac{n-m+2}{n-m+1}(n+1)^{m-2}.    
\end{align}
Substituting \eqref{subin} into \eqref{here} we obtain
\begin{equation}\label{mid}
    Q_{m,n,1/2}(1) = \frac{n-m+2}{(n-m+1)(n+1)} - \frac{(n-m+2)}{2(n+1)(n-m+1)} + \frac{n^{m-2}}{2(n+1)^{m-1}}.
\end{equation} 
Substituting \eqref{mid} into the right-hand side of \eqref{eq:lb1} gives us 
\begin{align*}
        \|Q_{m,n,1/2} - {\rm Uni}_n\|_{\TV} & \ge \frac{1}{2}\left|\frac{n-m+2}{(n-m+1)(n+1)} - \frac{(n-m+2)}{2(n+1)(n-m+1)} + \frac{n^{m-2}}{2(n+1)^{m-1}} - \frac{1}{n}  \right|\\
        & = \frac{1}{2}\left| \frac{n^{m-2}}{2(n+1)^{m-1}} - \frac{1}{2(n+1)}\left[1 + \frac{n-2m+2}{n(n-m+1)}\right] \right|,
    \end{align*}
which concludes this part of the proof.\\

\noindent \underline{Upper bound.} By Proposition \ref{prop:generalboundstv}, it is enough to consider the case $p=1$. By the definition of the total variation distance given in \eqref{def:dTV}, we have that
\begin{equation}\label{eq:dtv}
    \begin{split}
            \|Q_{m,n,1} - {\rm Uni}_n\|_{\TV} & = \frac{1}{2}\sum_{j=1}^n\left|Q_{m,n,1}(j) - \frac{1}{n}\right|.
    \end{split}
\end{equation}
By Theorem \ref{component} with $p=1$, we have that for any $j \in [n]$,
\begin{equation}\label{usethisone}
    \begin{split}
        Q_{m,n,1}(j) &= \frac{n-m+2}{(n-m+1)(n+1)}-\frac{1}{(n+1)^{m-1}} \Big[\sum_{s=n-j+1}^{m-1} \binom{m-1}{s} (n-s)^{m-s-2} (s+1)^{s-1} \Big].
    \end{split}
\end{equation}
Using the fact that
$$
\frac{n-m+2}{(n-m+1)(n+1)} - \frac{1}{n} = \frac{1}{(n-m+1)(n+1)} - \frac{1}{(n+1)n}
$$ 
along with \eqref{usethisone}, we have that \eqref{eq:dtv} satisfies
\begin{align}
        \|Q_{m,n,1} - {\rm Uni}_n\|_{\TV} & = \frac{1}{2}\sum_{j=1}^n\Big| \frac{1}{(n-m+1)(n+1)} - \frac{1}{(n+1)n}\notag \\
        & \qquad \qquad -\frac{1}{(n+1)^{m-1}} \Big[\sum_{s=n-j+1}^{m-1} \binom{m-1}{s} (n-s)^{m-s-2} (s+1)^{s-1} \Big]\Big| \notag\\
        & \hspace{-.5in} \le \frac{m-1}{2(n+1)(n-m+1)} \notag\\
        & \hspace{-.5in}\qquad \qquad + \frac{1}{2(n+1)^{m-1}}\sum_{j=1}^n\sum_{s=n-j+1}^{m-1} \binom{m-1}{s} (n-s)^{m-s-2} (s+1)^{s-1}.\label{eq:dtvupper}
    \end{align}
Now, interchanging the summands in \eqref{eq:dtvupper} and applying \eqref{1} and \eqref{2} gives us
\begin{align}
        &\sum_{j=1}^n\sum_{s=n-j+1}^{m-1} \binom{m-1}{s} (n-s)^{m-s-2} (s+1)^{s-1}\notag\\
        &\hspace{1in} = \sum_{s=0}^{m-1}\binom{m-1}{s} (n-s)^{m-s-2} (s+1)^{s-1}s\notag\\
        &\hspace{1in} = A_{m-1}(n-m+1,1;-1,0)-A_{m-1}(n-m+1,1;-1,-1)\notag\\
        &\hspace{1in} = \frac{(n+1)^{m-1}}{n-m+1}-\frac{(n-m+2)(n+1)^{m-2}}{n-m+1}\notag\\
        &\hspace{1in}=\frac{(m-1)(n+1)^{m-2}}{n-m+1},\label{eq:last}.
    \end{align}
Substituting \eqref{eq:last} into \eqref{eq:dtvupper} yields
\begin{equation}
        \|Q_{m,n,1} - {\rm Uni}_n\|_{\TV}  \le \frac{m-1}{2(n+1)(n-m+1)} + \frac{m-1}{2(n+1)(n-m+1)} =\frac{m-1}{(n+1)(n-m+1)},
\end{equation}
 which is enough to conclude the proof.
\end{proof}

Proposition \ref{cor:mcn} characterizes the convergence of $Q_{cn,n,p}$ to the uniform distribution ${\rm Uni}_n$ when the number of cars, $m$, is a proportion of the number of spots, $n$. In what follows, we analyze the behavior of $\|Q_{cn,n,p} - {\rm Uni}_n\|_{\TV}$ as a function of $c$, when $n$ and $p$ are treated as fixed parameters. More specifically, we analyze how $c$ may speed up or slow down the convergence of $Q_{cn,n,p}$ to the uniform distribution ${\rm Uni}_n$.

\begin{figure}[!ht]
    \centering
    \includegraphics[width=0.5\linewidth]{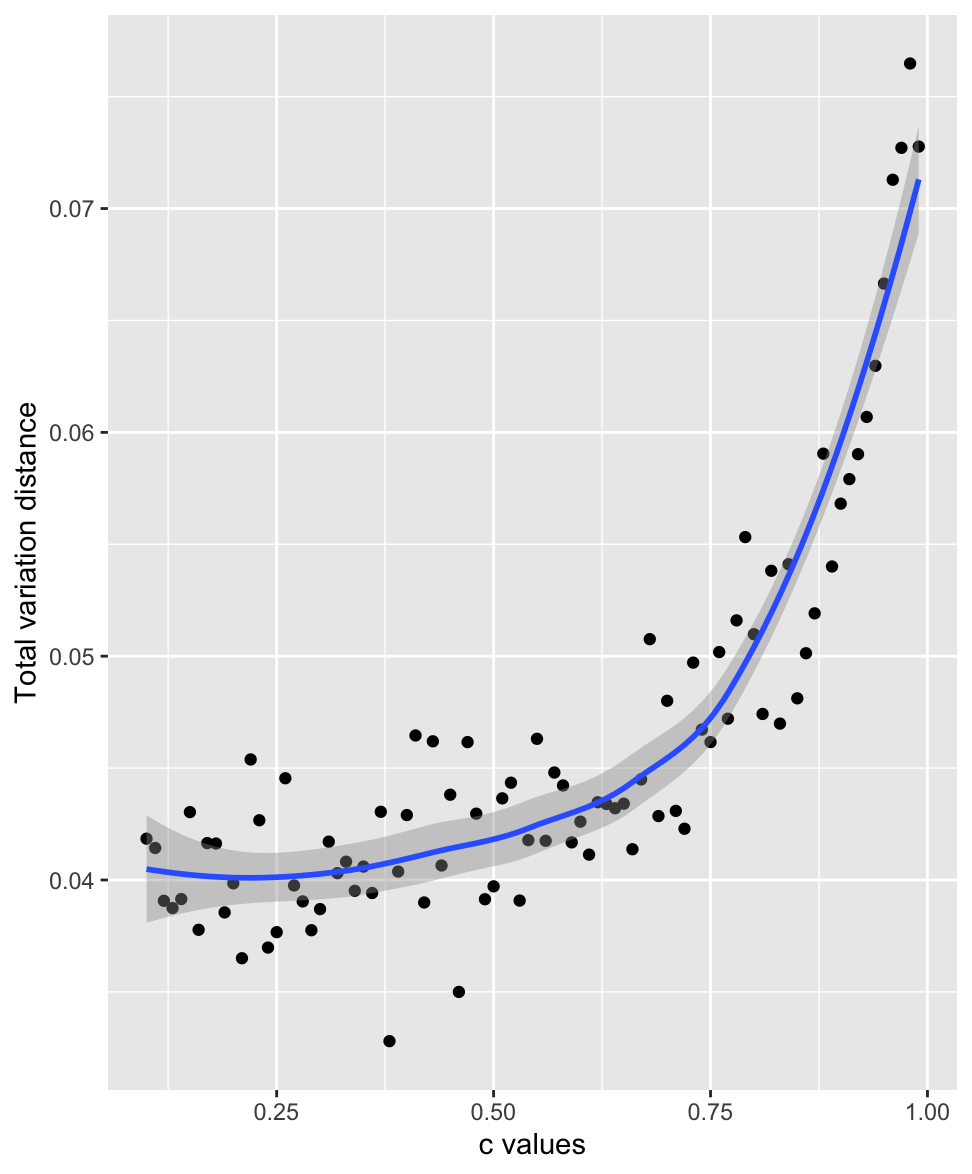}
    \caption{The total variation distance $\|Q_{cn,n,p} - {\rm Uni}_n\|_{\TV}$ as a function of $c$. Here $n=100$, $m=cn$, $p=1$, and $c$ ranges from $0.1$ to $0.99$.}
    \label{fig:enter-label}
\end{figure}
\begin{proposition}\label{prop:c} Take $m = cn$ for some $0<c<1$. Then, 
$$
    \sup_{n\ge 1}~\sup_{p\in [0,1]}\|Q_{cn,n,p} - {\rm Uni}_n\|_{\TV} \le 1-(1-c)e^{3c/5}.
$$
\end{proposition}

\begin{remark}
Here the coefficient $3/5$ may be further relaxed. All we need is a lower bound of $n \log(1+1/n)$ for all $n \geq 1$.
\end{remark}

\begin{remark}
For readers who are interested in investigating the dependence of $\|Q_{cn,n,p} -  \unif{n}\|_{\TV}$ as a function of $c$, we refer to our $R$ package on GitHub: \url{https://github.com/rbribeiro/parking-functions/}.
This $R$ package provides a set of functions to simulate and analyze parking behavior under varying probabilistic preferences.
\end{remark}

On the surface, Proposition \ref{prop:c} above may seem like a weaker version of Proposition \ref{cor:mcn}, however, the two propositions are different in spirit. Proposition \ref{prop:c} gives an upper bound on the function $c \mapsto \|Q_{cn,n,p} - {\rm Uni}_n\|_{\TV}$, regardless of the values of $n$ and $p$. In words, Proposition \ref{prop:c} states that, as a function of $c$, the total variation distance decreases to zero at least linearly in $c$, when $c$ goes to zero.
\begin{proof}[Proof of Proposition \ref{prop:c}]
By Proposition \ref{prop:generalboundstv}, it is enough to establish the result for the particular case $p=1$. For notational convenience, we write $\PP( \; \cdot \;)$ as a shorthand for $\PP(\; \cdot \; | \; \alpha \in [n]^m)$ and let $A$ be the following event
\begin{equation*}
    A \coloneqq  \{\alpha \in \PF{m,n}\}.
\end{equation*} 
Since $a_m$ is uniformly distributed over $[n]$ under $\PP$, we have that for any $j \in [n]$,
\begin{equation}
    \begin{split}
        \frac{1}{n} = \PP(a_m = j) & = \PP(a_m = j \; |\; A)\PP(A) + \PP(a_m = j \; |\; A^c)\PP(A^c)\\
        & = Q_{m,n,1}(j)\PP(A) + \PP(a_m = j \; |\; A^c) (1-\PP(A)),
    \end{split}
\end{equation} 
which implies that for any $j\in [n]$,
\begin{equation}\label{eq:end}
    \begin{split}
      Q_{m,n,1}(j) - \frac{1}{n} & = Q_{m,n,1}(j)(1-\PP(A)) - \PP(a_m = j \; |\; A^c) (1-\PP(A))\\
      & = \left[Q_{m,n,1}(j) - \PP(a_m = j \; |\; A^c) \right](1-\PP(A)).
      \end{split}
\end{equation}
Let $\widetilde{Q}_{m,n,1}(\; \cdot \;)$ be $\PP(a_m = \cdot \; |\; A^c)$ and recall the definition of the total variation distance given in \eqref{def:dTV}. By \eqref{eq:end}, we have that
\begin{equation}
    \|Q_{m,n,1} - \unif{n} \|_{\TV} = (1-\PP(A))  \|Q_{m,n,1} - \widetilde{Q}_{m,n,1} \|_{\TV}.
\end{equation}
On the other hand, by Theorem \ref{thm:old}, we have that
$$
\PP(A) = \frac{(n-m+1)(n+1)^{m-1}}{n^m} = \frac{(1-c)n+1}{n+1}\left( 1 + \frac{1}{n}\right)^{cn} \ge  (1-c)e^{3c/5},
$$
which proves the result, since $\|Q_{m,n,1} - \widetilde{Q}_{m,n,1} \|_{\TV} \le 1$.
\end{proof}

\end{document}